\documentclass[reqno]{amsart}
\usepackage{hyperref}
\def\N{{\mathbb{N}}}

\def\R{{\mathbb{R}}}

\def\T{{\mathbb{T}}}

\def\LL{{\mathcal{L}}}

\theoremstyle{plain}
\newtheorem{theorem}{Theorem}
\newtheorem{proposition}{Proposition}
\newtheorem{definition}{Definition}
\newtheorem{lemma}{Lemma} 
\newtheorem{corollary}{Corollary}
\theoremstyle{remark}
\newtheorem{remark}{Remark}
\newtheorem{example}{Examples}

\title[Metrization of probabilistic metric spaces ]{Metrization of probabilistic metric spaces. Applications to fixed point theory and Arzela-Ascoli type theorem}
\author{Mohammed Bachir, Bruno Nazaret}
\begin{document}

\date{05/06/2019} 
%46S50, 47S50
\address{Laboratoire SAMM 4543, Universit\'e Paris 1 Panth\'eon-Sorbonne\\
Centre P.M.F. 90 rue Tolbiac\\
75634 Paris cedex 13\\
France}

\email{Mohammed.Bachir@univ-paris1.fr}
\email{Bruno.Nazaret@univ-paris1.fr}
\begin{abstract}
Schweizer, Sklar and Thorp proved in 1960 that a Menger space $(G,D,T)$ under a continuous $t$-norm $T$, induce a natural topology $\tau$ wich is metrizable. We extend this result  to any probabilistic metric space $(G,D,\star)$  provided that the triangle function $\star$ is continuous. We prove in this case, that  the topological space $(G,\tau)$  is uniformly homeomorphic to a (deterministic) metric space $(G,\sigma_D)$ for some canonical metric $\sigma_D$ on $G$. As applications, we extend the fixed point theorem of Hicks to probabilistic metric spaces which are not necessarily Menger spaces and we prove a probabilistic Arzela-Ascoli type theorem. 
\end{abstract}
\maketitle
\vskip5mm
{\bf msc:} 54E70,  46S50.
%\tableofcontents

\section{\bf Introduction}
%The general concept of probabilistic metric spaces was introduced by K. Menger, who dealt with probabilistic geometry \cite{M1}, \cite{M2}, \cite{M3}. The decisive influence on the development of the theory of probabilistic metric spaces is due to B. Schweizer and A. Sklar and their coworkers in several papers \cite{SS0}, \cite{SS1}, \cite{SS2}, \cite{SS3}, see also \cite{SHE}, \cite{SHE1} \cite{KMP2}, \cite{HP} and \cite{HP1}. For more informations about this theory we refeer to the excellent monograph \cite{S.S}. 

%\vskip5mm

%Recently, the first author introduced in \cite{Ba} a natural concept of probabilistic Lipschitz maps defined from a probabilistic metric space $G$ into the set of all cumulative distribution functions that vanish at $0$, classically denoted by $\Delta^+$. In particular, the introduction of the space of all probabilistic $1$-Lipschitz maps provides a new method for the completion of probabilistic metric spaces extending a result of H. Sherwood in \cite{SHE1}. It also leads to a probabilistic version of the Banach-Stone theorem (see for instance \cite{Ba} and \cite{Ba2}).

\vskip5mm

Let $(G,D,T)$ be a Menger space equipped with a probabilistic metric $D$ and a $t$-norm $T$  (the definitions and notation reminders will be given in the details in Section \ref{S1}). Schweizer and Sklsar \cite{SS1} defined for $\varepsilon, \lambda> 0$  and each $x \in G$ a neighborhood $N_x(\varepsilon, \lambda)$ as follows
\[
N_x(\varepsilon, \lambda)=\lbrace y\in G: D(x,y)(\varepsilon)>1-\lambda \rbrace.
\]
Schweizer, Sklar and Thorp proved in \cite{SS3} that, given a $t$-norm $T$ of a Menger space $(G,D,T)$ satisfying $1 = \sup_{x<1} T(x, x)$ (in particular if $T$ is continuous), the collection $\lbrace N_x(\varepsilon, \lambda): x\in G \rbrace$ taken as a neighborhood base at $x$ gives rise to a metrizable topology.  In \cite{MN} Morrel and Nagata proved the following two extensions:
\begin{enumerate}
	\item The class of topological Menger spaces coincides with that of semi-metrizable topological spaces.
	\item No condition on $T$ weaker than $1 = \sup_{x<1} T(x,x)$ can guarantee that a Menger space, under $T$, is topological.
\end{enumerate}

The aim of the present paper is to prove that, in a general probabilistic metric space $(G,D,\star)$, not necessarily being a Menger space, the collection $\lbrace N_x(\varepsilon, \lambda): x\in X \rbrace$ taken as a neighborhood base at $x$ gives rise to a  topology which is uniformly homeomorphic to a metric space, provided that the triangle function $\star$ is continuous (necessarily uniformly continuous by Sibley's result in \cite{DS} on the compactness of $(\Delta^+,d_\LL)$, where $d_\LL$ denotes the modified L\'evy distance and $\Delta^+$ denotes the set of all nondecreasing and left-continuous distributions that vanish at $0$).

We get an even more precise result : if  $w_\star : [0,+\infty]\to [0,+\infty]$ is a modulus of uniform continuity for the triangle function $\star$, then the (deterministic) metric $\sigma_D$ on $G$ defined canonically from the probabilistic metric $D$ by 
\[
\forall x, y\in G, \hspace{2mm} \sigma_D(x,y):=\sup_{z\in G} d_\LL(D(x,z), D(z,y)),
\]
satisfies the following inequalities
\begin{equation}\label{estimate_cont}
\forall x, y\in G, \ d_\LL(D(x,y),\mathcal{H}_0)\leq \sigma_D(x,y)\leq w_\star(d_\LL(D(x,y),\mathcal{H}_0)).
\end{equation}
Note from \cite{S.S} that $ y\in N_x(t,t)$ if and only if $d_\LL(D(x,y), \mathcal{H}_0) < t$, for all $t>0$.
If moreover we assume that $\star$ is $k$-Lipschitz given some positive real number $k$ (necessarily $k\geq 1$), then we can take $w_\star(t)=kt$ for all $t\geq 0$ (see Proposition \ref{Tri} for examples of such functions).
% and so we have $$(A2) \hspace{2mm}\forall x, y\in G, \hspace{2mm} d_\LL(D(x,y),\mathcal{H}_0)\leq \sigma_D(x,y)\leq k d_L(D(x,y),\mathcal{H}_0).$$
As an immediate consequence of \eqref{estimate_cont}, the semi-metric $\alpha(x,y):=d_\LL(D(x,y),\mathcal{H}_0)$ define a topology on $(G,D,\star)$ which is uniformly (resp. Lipschitz) homeomorphic to  the metric space $(G,\sigma_D)$, whenever $\star$ is continuous (resp. Lipschitz continuous). This result is an extension to non necessarily Menger spaces of the works established in Menger spaces by Schweizer, Sklar and Thorp in \cite{SS3}. In particular, the formula \eqref{estimate_cont} allows us to transfer several known results from  metric space theory to the probabilistic metric theory. For instance, using \eqref{estimate_cont} and the Ekeland variational principle we give some extensions of the fixed point theorem of Hicks (see \cite{HP}), or using again \eqref{estimate_cont} we give an Arzela-Ascoli type theorem for the the space of probabilistic $1$-Lipschtz maps introduced recently in \cite{Ba}. Notice that other results such as Baire theorem and all its variants/consequences can be transfered, thanks to our result, to the probabilistic metric framework.

This paper is organized as follows. In Section \ref{S1}, we recall some classical notions related to probabilistic metric space. In Section \ref{S2}, we treat the metrization of probabilistic metric space and prove Theorem \ref{Metrizable}. We also give some new properties. In Section \ref{S3}, we establish fixed point theorems (Theorem \ref{fix1} and Theorem \ref{fix2}) extending a result of Hicks (see \cite{HP}). In Section \ref{S4}, we prove Theorem \ref{AA}, showing that the set of  probabilistic $1$-Lipschitz maps  introduced in \cite{Ba} is a compact space for the uniform convergence, giving a probabilistic Arzela-Ascoli theorem.

\section{Definitions and notation} \label{S1}
In this section, we recall some known facts about probabilistic metric spaces, the modified L\'evy distance and the weak convergence. All these notions can be found in \cite{S.S}, \cite{KMP2} and \cite{KMP3}. We also recall the notion of probabilistic $1$-Lipschitz map introduced in \cite{Ba}, which shall play an important role in the sequel.

\subsection{Probabilistic metric space and triangle function}
By $\Delta^+$ we denote the set of all (cumulative) distribution functions $F : [-\infty, +\infty] \longrightarrow [0, 1]$, nondecreasing and left-continuous with $F(-\infty) = 0$; $F(+\infty) = 1$ and $F(0) = 0$. For $a\in [0,+\infty[$, we denote $\mathcal{H}_a(t)=0$ if $t\leq a$ and  $\mathcal{H}_a(t)=1$, if $t>a$.

In the sequel, we shall write $F\leq G$ for
\[
\forall t\in\R, \quad F(t)\leq G(t),
\]
which defines an ordering relation on $\Delta^+$.

\begin{definition} \label{axiom1} (\cite{S.S,HP, KMP2, KMP3})  A binary operation $\star$ on $\Delta^+$ is called a triangle function if and only if it is  commutative, associative, non-decreasing in each place, and has $\mathcal{H}_0$ as neutral element. In other words:
\begin{itemize}
\item[(i)] $F\star L \in \Delta^+$ for all $F, L \in \Delta^+$.
\item[(ii)] $F\star L=L\star F$ for all $F, L \in \Delta^+$.
\item[(iii)] $F\star(L\star K)=(F\star L)\star K$, for all $F,L,K\in \Delta^+$.
\item[(iv)] $F\star\mathcal{H}_0=F$ for all $F\in \Delta^+$.
\item[(v)] $F\leq L \Longrightarrow F\star K\leq L\star K$ for all $F, L, K\in \Delta^+$.
\end{itemize}
\end{definition}
\begin{definition} A $t$-norm is a function $T:[0,1]\times[0,1]\to[0,1]$, usually called a triangular norm (see  \cite{S.S,HP, KMP2, KMP3}),  satisfying
\begin{itemize}
\item $T(x,y) = T(y,x)$ ( commutativity);
\item $T(x,T(y,z)) = T(T(x,y),z)$ (associativity);
\item $T(x,y) \leq T(x,z)$ whenever $y\leq z$ (monotonicity );
\item $T(x, 1) = x$ (boundary condition).	
\end{itemize}
\end{definition}
\begin{definition} A probabilistic metric space $(G,D,\star)$ (an PM-space) is a set $G$ together with a triangle function $\star$ and a function $D: G\times G\to \Delta^+$ satisfying:
\begin{itemize}
\item[(i)] $D(x,y)=\mathcal{H}_0$ iff $x=y$.
\item[(ii)] $D(x,y)=D(y,x)$ for all $x,y\in G$ %(Symmetry).
\item[(iii)] $D(x,y)\star D(y,z)\leq D(x,z)$ for all $x, y, z\in G$ %(Triangulzr inequality).
\end{itemize}
%$A1.$ $F_{xy}(0)=0$,

%$A2.$ $F_{xy}=F_{yx}$ for all $x, y \in X$,

%$A3.$  $F_{xy}(t)=0$ for all $t>0$, iff $x=y$,

%$A4.$ if $F_{xy}(r)=1$ and $F_{yz}(s)=1$, then $ F_{xz}(r+s)=1$.
\end{definition}
Usually, $D(x,y)$ is denoted by $F_{x,y}$ in the literature.  A probabilistic metric space $(G,D,\star)$ is called a Menger space and denoted by $(G,D,T)$, iff the triangle function $\star:=\star_T$ is defined from a $t$-norm $T$ as follows: for all $F, L\in \Delta^+$ and for all $t\in \R$,
\begin{eqnarray} \label{eq0}
(F\star_T L)(t) &:=& \sup_{u+v=t} T(F(u),L(v))\\
                        &=&  \sup_{u,v\leq 0: u+v=t} T(F(u), L(v))\nonumber
\end{eqnarray}

\subsection{ L\'evy distance and weak convergence}
\begin{definition}
Let $F$ and $G$ be in $\Delta^+$. For any $h>0$ we set
\[
A_{F,G}^h=\left\{t\geq 0 \text{ st. } G(t)\leq F(t+h)+h\right\}.
\]
The modified L\'evy distance is the map $d_\LL$ defined on $\Delta^+ \times \Delta^+$ as
$$d_\LL(F, G) = \inf \left\lbrace h>0 \text{ st. } [0,h^{-1}[\subset A_{F,G}^h\cap A_{G,F}^h \right\rbrace.$$
\end{definition}
Notice that, for all $F$, $G\in\Delta^+$,
\begin{itemize}
	\item[(i)] if $F\leq G$ then $A_{G,F}=[0,+\infty[$, hence
	\[
	d_\LL(F, G) = \inf \left\lbrace h>0 \text{ st. } [0,h^{-1}[\subset A_{F,G}^h\right\rbrace.
	\]
	\item[(ii)] if $h\geq 1$, $A_{F,G}^h=A_{G,F}^h=[0,+\infty[$, hence $d_\LL(F, G)\leq 1$.
	\item[(iii)] The usual Levy distance between general cumulative distribution functions can be expressed as
	\[
	\inf \left\lbrace h>0 \text{ st. } A_{F,G}^h = A_{G,F}^h = [0,+\infty[ \right\rbrace.
	\]
	It is invariant under the action of translations which, as we shall see later, is not the case for the modified version since it somehow does not see the behaviour at infinity.
\end{itemize}

\begin{definition} \label{contin} Let $\star$ be a triangle function on $\Delta^+$.

$(1)$ A sequence $(F_n)$ of distributions in $\Delta^+$ converges weakly to a function $F$ in $\Delta^+$ if $(F_n(t))$ converges to $F(t)$ at each point $t$ of continuity of $F$. In this case, we write indifferently $F_n \,{\xrightarrow {\textnormal{w}}}\, F$ or $\lim_n F_n =F$.

$(2)$ We say that the law $\star$ is continuous at $(F,L)\in \Delta^+\times \Delta^+$ if we have $F_n\star L_n \,{\xrightarrow {\textnormal{w}}}\, F\star L$, whenever $F_n \,{\xrightarrow {\textnormal{w}}}\, F$ and $L_n \,{\xrightarrow {\textnormal{w}}}\, L$. 
\end{definition}
%Recall from \cite[Lemme 4.3.4]{S.S} and \cite{SS4} that the map $F\mapsto d_\LL(F,\mathcal{H}_0)$ is non-increassing, that is
%\begin{eqnarray*}
%F,G\in \Delta^+, F\leq G\Longrightarrow d_\LL(G,\mathcal{H}_0) \leq d_\LL(F,\mathcal{H}_0).
%\end{eqnarray*}
We recall the following results due to D. Sibley in \cite[Theorem 1. and Theorem 2]{DS}.
\begin{lemma} \label{DS1} (\cite{DS,S.S}) The function $d_\LL$ is a metric on $\Delta^+$ and $(\Delta^+,d_\LL)$ is compact.
\end{lemma}
\begin{lemma} \label{DS2} (\cite{DS,S.S}) Let $(F_n)$ be a sequence of functions in $\Delta^+$, and let $F$ be
an element of $\Delta^+$. Then $(F_n)$ converges weakly to $F$ if and only if $d_\LL(F_n, F)\longrightarrow 0$, when $n \longrightarrow +\infty$.
\end{lemma}
\begin{remark} Thanks to Lemma \ref{DS2}, we shall indifferently use the notations $F_n\,{\xrightarrow {\textnormal{w}}}\, F$ or $d_\LL(F_n, F)\longrightarrow 0$ to say that $(F_n)$ converges weakly to $F$.
\end{remark}
\begin{definition} \label{neigh}
Let $(G,D,\star)$ be a probabilistic metric space. For $x \in G$ and $t > 0$, the strong $t$-neighborhood
of $x$ is the set
$$N_x(t) = \lbrace y \in G : D(x,y)(t) > 1 - t \rbrace,$$
and the strong neighborhood system for $G$ is $\lbrace N_x(t); x\in G, t > 0 \rbrace.$
\end{definition}
\begin{lemma} \label{Topo} (\cite[Lemme 4.3.3]{S.S}) Let $t > 0$ and $x, y \in G$. Then we have
$ y\in N_x(t)$ if and only if $d_\LL(D(x,y), \mathcal{H}_0) < t$.
\end{lemma}

%\begin{proposition} \label{Oujda} \textnormal{ (\cite[Theorem 2.2, Theorem 2.3]{MON})}  Let $(K,D,\star)$  be a complete probabilistic metric space. Then, we have:

%$(1)$  $(K,D,\star)$  is compact if and only if every sequence has a convergent subsequence.

%$(2)$ If $(K,D,\star)$ is compact, then  $(K,D,\star)$  is separable. 

%\end{proposition}
%%
%%
\subsection{Probabilistic $1$-Lipschitz map}

\begin{definition}\label{Ldef2}  Let $(G, D, \star)$ be a probabilistic metric space and let $f$ be a function $f : (G, D, \star) \longrightarrow (\Delta^+,d_\LL)$. We say that $f$ is a probabilistic $1$-Lipschitz map if : 
$$\forall x, y \in G,\hspace{1mm} D(x,y)\star f(y)\leq f(x).$$ 
\end{definition}
\vskip5mm
 We can also define probabilistic $k$-Lipschitz maps for any nonegative real number $k\geq 0$ as the maps $f$ satisfying 
 $$\forall x, y \in G,\hspace{1mm} D_k(x,y)\star f(y)\leq f(x),$$
where, for all $x, y \in G$ and all $t\in \R$, $D_k(x,y)(t)=D(x,y)(\frac{t}{k})$ if $k>0$ and $D_0(x,y)(t)=\mathcal{H}_0(t)$ if $k=0$. For sake of simplicity, when we use the notion in Definition \ref{Ldef2}, we shall only treat in this paper the case of probabilistic $1$-Lipschitz maps, but our main result result could be easily extended to this more general setting.
\begin{example} \label{example.2} Let $(G,d)$ be a metric space. Assume that $\star$ is a triangle function on $\Delta^+$ satisfying $\mathcal{H}_a\star \mathcal{H}_b=\mathcal{H}_{a+b}$ for all $a, b\in \R^+$ (for example if $\star=\star_T$ where $T$ is a lef-continuous triangular norm). Let $(G, D, \star)$ be the probabilistic metric space defined with the probabilistic metric 
$$D(p,q)=\mathcal{H}_{d(p,q)}.$$
Let $L: (G,d) \longrightarrow \R^+$ be a real-valued map. Then, $L$ is a non-negative $1$-Lipschitz map if and only if $f : (G,D,\star) \longrightarrow \Delta^+$ defined for all $x\in G$ by
$$f(x):=\mathcal{H}_{L(x)}$$
is a probabilistic $1$-Lipschitz map. This example shows that the framework of probabilistic $1$-Lipschitz maps encompasses the classical determinist case.

\end{example}
By $Lip^1_\star(G,\Delta^+)$ we denote the space of all probabilistic $1$-Lipschitz maps   
$$Lip^1_\star(G,\Delta^+):=\lbrace f : G\longrightarrow \Delta^+/ D(x,y)\star f(y)\leq f(x); \forall x, y \in G\rbrace.$$
For all $x\in G$, by $\delta_x$ we denote the map 
\begin{eqnarray*}
\delta_x : G &\longrightarrow& \Delta^+\\
            y&\mapsto& D(y,x).
\end{eqnarray*}

It follows from the properties of the probabilistic metric $D$ that $\delta_x$ is a probabilistic $1$-Lipschitz for every $x\in G$. We set $\mathcal{G}(G):=\lbrace \delta_x, x\in G \rbrace$ and by $\delta$, we denote the operator  
\begin{eqnarray*}
\delta : G &\longrightarrow& \mathcal{G}(G)\subset Lip^1_\star(G,\Delta^+)\\
            x &\mapsto& \delta_x.
\end{eqnarray*}
\subsection{Modulus of uniform continuity of a triangle function on $\Delta^+$}
Let $\star:\Delta^+\times \Delta^+ \to \Delta^+$ be a continuous triangle function (with respect to the modified L\'evy distance $d_\LL$). Since $(\Delta^+,d_\LL)$ is a compact metric space (see Lemma \ref{DS1}) and $\star$ is continuous, then $\star$ is uniformly continuous from $\Delta^+\times \Delta^+$ into $\Delta^+$. Let $\omega_\star:[0,+\infty]\to [0,+\infty]$ be a modulus of uniform continuity for $\star$ ($\lim_{t\to 0}\omega_\star(t)=\omega_\star(0)=0$), that is 
%for all $t\geq 0$, 
%\begin{eqnarray} \label{eqs1} 
%\omega_\star (t):= \sup\lbrace d_\LL(F\star L, F'\star L'): F,L, F',L'\in \Delta^+; d_\LL(F,F')+d_\LL(L,L')=t\rbrace.
%\end{eqnarray}
%Then, we have that
for all $(F,L), (F',L')\in \Delta^+\times \Delta^+$
$$d_\LL(F\star L, F'\star L')\leq \omega_\star(d_\LL(F,F')+ d_\LL(L,L')).$$
In particular for all $(F,L) \in \Delta^+\times \Delta^+$
\begin{eqnarray} \label{eqs2} 
d_\LL(F\star L, L)\leq \omega_\star(d_\LL(F,\mathcal{H}_0)).
\end{eqnarray}
If moreover the operation $\star$ is $k$-Lipschitz (with respect to $d_\LL$) for some positive number $k$ then $\omega_\star(t)=k t$ for all $t\geq 0$ (necessarily $k\geq 1$, by using \ref{eqs2} with $L=\mathcal{H}_0$) is a modulus of uniform continuity. We give in Proposition \ref{Tri} examples of $k$-Lipschitz triangle function using $k$-Lipschitz $t$-norms. 

\section{ Metrization of  Probabilistic Metric space.}\label{S2}
%%
%%
%%%%%%%%%%%%%%%%%%%%%%%%%
%%%%%%%%%%%%%%%%%%%%
We give below the main result of this section, that is a metrization of probabilistic metric space extending the result of Schweizer, Sklar and Thorp in \cite{SS3}. 
\vskip5mm
Let $(G, D, \star)$ be a probabilistic metric space. We define canonically the metric $\sigma_D$ on $G$ using the probabilistic metric $D$ as follows: for all $x, y \in G$
\begin{equation*} \sigma_D (x,y)  := \sup_{z\in K} d_\LL(D(x,z), D(y,z)):=\sup_{z\in K} d_\LL(\delta_x(z), \delta_y(z)):=d_{\infty}(\delta_x,\delta_y)
\end{equation*}
It is easy to see that $\sigma_D$ is a metric on $G$ and that for all $x, y\in G$
\begin{equation*} \label{In} d_\LL(D(x,y), \mathcal{H}_0) \leq  \sigma_D (x,y).
\end{equation*}
\begin{theorem} \label{Metrizable} Let $(G, D, \star)$ be a probabilistic metric space such that $\star$ is continuous (resp. $k$-lipschitz). Let $\omega_\star$ be a modulus of uniform continuity of $\star$ on $\Delta^+$. Then, the metric $\sigma_D$ satisfies: for all $x, y \in G$
$$d_\LL(D(x,y), \mathcal{H}_0) \leq  \sigma_D (x,y)\leq \omega_\star(d_\LL(D(x,y),\mathcal{H}_0)).$$
$$(\textnormal{resp. } d_\LL(D(x,y), \mathcal{H}_0) \leq  \sigma_D (x,y)\leq kd_\LL(D(x,y),\mathcal{H}_0)).$$
In particular, the identity map $i: (G, \tau)\to (G,\sigma_D)$ is  an uniform homeomorphism, where $\tau$ is the topology induced by the  strong neighborhood system $\lbrace N_x(t); x\in G, t > 0 \rbrace$ (see Definition \ref{neigh}). 
\end{theorem}
This theorem is a mere consequence of the following lemma, that we will also use for proving Theorem~\ref{AA}.
\begin{lemma} \label{equi} Let $(G, D, \star)$ be a probabilistic metric space such that $\star$ is continuous.  Let $\omega_\star$ be a modulus of uniform continuity of $\star$ on $\Delta^+$. Then, the set $Lip_\star^1(G,\Delta^+)$ is uniformly  equicontinuous. More precisely, we have  $\forall x, y\in G:$
$$\sup_{f\in Lip_\star^1(G,\Delta^+)}d_\LL(f(x),f(y)) \leq \omega_\star(d_\LL(D(x,y),\mathcal{H}_0)).$$
\end{lemma}
\begin{proof} From the formula (\ref{eqs2}) about the modulus of uniform continuity of $\star$, we have that $ \forall L\in \Delta^+, \forall x, y\in G:$
$$ d_\LL(D(x,y)\star L,L) \leq \omega_\star(d_\LL(D(x,y),\mathcal{H}_0)).$$
In particular, we have for all $f\in Lip_\star^1(G,\Delta^+)$ and all $x$, $y\in G$,
\[
\max [d_\LL(D(x,y)\star f(x),f(x)),d_\LL(D(x,y)\star f(y),f(y))]\leq \omega_\star(d_\LL(D(x,y),\mathcal{H}_0)),
\]
hence, it is enough to prove that 
\[
d_\LL(f(x),f(y)) \leq \max [d_\LL(D(x,y)\star f(x),f(x)),d_\LL(D(x,y)\star f(y),f(y))].
\]
Let $h_1, h_2 > 0$ such that, 
\begin{equation}\label{eq:2}
[0,h_1^{-1}[\subset A^{h_1}_{D(x,y)\star f(x),f(x)}\cap A^{h_1}_{f(x),D(x,y)\star f(x)}.
\end{equation}
\begin{equation}\label{eq:21}
[0,h_2^{-1}[\subset A^{h_2}_{D(x,y)\star f(y),f(y)}\cap A^{h_2}_{f(y),D(x,y)\star f(y)}.
\end{equation}
that is, for all $t\in ]0,h_1^{-1}[$ and all $t'\in ]0,h_2^{-1}[$, we have
\begin{eqnarray*}
0\leq D(x,y)\star f(x)(t) &\leq & f(x)(t+h_1) + h_1 \\
0\leq f(x)(t) &\leq & D(x,y)\star f(x)(t+h_1) + h_1 \\
0\leq D(x,y)\star f(y)(t') &\leq & f(y)(t'+h_2) + h_2\\
0\leq f(y)(t') &\leq & D(x,y)\star f(y)(t'+h_2) + h_2. 
\end{eqnarray*}
From the second, the fourth inequalities and the fact that $f$ is $1$-Lipschitz, we get that for all $t\in ]0,h_1^{-1}[$ and all $t'\in ]0,h_2^{-1}[$
\begin{eqnarray*}
0\leq f(x)(t) \leq  f(y)(t+h_1) + h_1 \\
0\leq f(y)(t') \leq  f(x)(t'+h_2) + h_2.
\end{eqnarray*} 
It follows that for all $s\in ]0, \max(h_1,h_2)^{-1}[$ (a subset of $]0,\min(h_1^{-1},h_2^{-1})[$)
\begin{eqnarray*}
0\leq f(x)(s) \leq  f(y)(s+ \max(h_1,h_2)) + \max(h_1,h_2) \\
0\leq f(y)(s) \leq  f(x)(s+ \max(h_1,h_2)) + \max(h_1,h_2).
\end{eqnarray*}
Thus, we have that $d_\LL(f(x),f(y))\leq \max(h_1,h_2)$ for all $h_1, h_2 > 0$ satisfying \eqref{eq:2} and \eqref{eq:21}. This implies that 
\begin{eqnarray*}
d_\LL(f(x),f(y)) \leq \max (d_\LL(D(x,y)\star f(x),f(x)),d_\LL(D(x,y)\star f(y),f(y))),
\end{eqnarray*}
and the conclusion.
\end{proof}
Let us now prove Theorem \ref{Metrizable}.
\begin{proof}[Proof of Theorem \ref{Metrizable}] The inequality at the left is a direct consequence of the definition of $\sigma_D$. To  prove the inequality at the right, we use Lemma \ref{equi} noticing that
\[
\mathcal{G}(G):=\lbrace \delta_x/ x\in G \rbrace\subset  Lip_\star^1(G,\Delta^+).
\]
The second part  of the theorem follows from Lemma \ref{Topo} since, $y\in N_x(t)$ if and only if $d_\LL(D(x,y), \mathcal{H}_0) < t$ for each $t > 0$ and $x, y \in G$.
\end{proof}
%%%%%%%%%%
%%%%%%%%%%
The notion of probabilistic distance naturally leads to associated metric concepts, such as Cauchy sequence, completeness, separability, density and compatness.
\begin{definition} A complete probabilistic metric space $(K,D,\star)$ is called compact if for all $t>0$, the
open cover $\lbrace N_x(t): x\in K \rbrace$ has a finite subcover.
\end{definition}
%This notion of probabilistic distance naturally leads to associated metric concepts, such as Cauchy sequence and completeness.
\begin{definition} In a probabilistic metric space $(G, D, \star)$, a sequence $(z_n)\subset G$ is said to be a Cauchy sequence if for all $t\in \R$,
$$\lim_{n,p \longrightarrow +\infty} D(z_n, z_p)(t)=\mathcal{H}_0(t).$$
(Equivalently, if $D(z_n, z_p)\,{\xrightarrow {\textnormal{w}}}\,\mathcal{H}_0$ or $d_\LL(D(z_n, z_p),\mathcal{H}_0)\to 0$, when $n, p\longrightarrow +\infty$). A probabilistic metric space $(G, D, \star)$ is said to be complete if every Cauchy sequence $(z_n)\subset G$ weakly converges to some $z_{\infty}\in G$, that is $\lim_{n \rightarrow +\infty} D(z_n, z_{\infty})(t)=\mathcal{H}_0(t)$ for all $t\in \R$, we will briefly note $\lim_n D(z_n, z_{\infty})=\mathcal{H}_0$. 
\end{definition}
\begin{corollary} \label{passage} Let $(G, D, \star)$ be a probabilistic metric space such that $\star$ is continuous. Then, the following assertions hold.
	
	$(1)$ $(G, D, \star)$ is a probabilistic complete metric space iff $(G,\sigma_D)$ is a complete metric space.
	
	$(2)$ $(G, D, \star)$ is compact as probabilistic metric space iff $(G,\sigma_D)$ is a compact metric space.
	
	$(3)$ $(G, D, \star)$ is separable as probabilistic metric space iff $(G,\sigma_D)$ is separable metric space.
	
\end{corollary}
\begin{proof} It is a direct consequence of  Theorem \ref{Metrizable} using Lemma \ref{Topo}.
\end{proof}

 Notice that several results in the litterature proved for probabilistic metric spaces could be easily deduced from Corollary \ref{passage} and Theorem \ref{Metrizable}. For instance, recall that a Baire space is a topological space such that every intersection of a countable collection of open dense sets is also dense. In \cite{SHE0},  H. Sherwood proved  that a complete Menger space under a continuous $t$-norm, equipped with the topology $\tau$ induced by the  strong neighborhood system $\lbrace N_x(t); x\in G, t > 0 \rbrace$  is a Baire space. Now, Theorem \ref{Metrizable} expressing the fact that as soon as the triangle function is continuous then the induced topology is metrizable, we immediately obtain the following result.
\begin{proposition}  Let $(G, D, \star)$ be a probabilistic complete metric space such that $\star$ is continuous. Let $\tau$ be the  topology induced by strong neighborhood system $\lbrace N_x(t); x\in G, t > 0 \rbrace$ (see Theorem \ref{Metrizable}). Then, $(G, \tau)$ is a Baire space.
\end{proposition}
In the same spirit, we also easily recover the following proposition already proven by other means in \cite[Theorem 2.2, Theorem 2.3]{MON}.
\begin{proposition} Let $(K,D,\star)$ be a probabilistic metric space. Suppose that the triangle function $\star$ is continuous. Then, 
	
	$(1)$ $(K, D, \star)$ is compact as probabilistic metric space iff every sequence of $K$ has a convergent subsequence.
	
	$(2)$ If $(K, D, \star)$ is compact as probabilistic metric space, then it is separable.
\end{proposition}

We end the section by showing that the metric $\sigma_D$ is canonical in the following sens. We know that every (complete) metric space induce a probabilistic (complete) metric space. Indeed, if $d$ is a (complete) metric on $G$ and $\star$ is a triangle function on $\Delta^+$ satisfying $\mathcal{H}_a\star \mathcal{H}_b=\mathcal{H}_{a+b}$ for all $a, b\in \R^+$  (see references \cite{S.S} and \cite{HP}), then $(G, D, \star)$ is a probabilistic (complete) metric space, where
$$D(p,q)=\mathcal{H}_{d(p,q)}, \hspace{1mm} \forall p,q \in G.$$ 

\vskip5mm
Using  Proposition \ref{H} below, we get that 
\begin{eqnarray*}
d_\LL(D(p,q), \mathcal{H}_ {0})\leq \sigma_D(p,q):=\sup_{z\in G} d_\LL(D(p,z), D(z,q))&=&\sup_{z\in G} d_\LL(\mathcal{H}_{d(p,z)}, \mathcal{H}_ {d(z,q)})\\
                                                                               &\leq& \sup_{z\in G} \min(1, |d(p,z)-d(z,q)|)\\
                                                                               &=& \min(1, d(p,q))\\
                                                                               &=& d_\LL(\mathcal{H}_{d(p,q)}, \mathcal{H}_ {0})\\
                                                                               &=&  d_\LL(D(p,q), \mathcal{H}_ {0})
\end{eqnarray*}
Thus, we have the equality
\[
d_\LL(D(p,q), \mathcal{H}_ {0})=\sigma_D(p,q)=\min(1, d(p,q)).
\]
It follows that $\sigma_D(p,q)=d(p,q)$, for all $p,q\in G$ such that $d(p,q)\leq 1$. In particular, $\sigma_D$ and $d$ coincides if $(G,d)$ is of diameter less than $1$.

%%%%%%%%%
%%%%%%%%%
\begin{proposition} \label{H} Let $a, b \geq 0$. Then,
\[
d_\LL(\mathcal H_a,\mathcal H_b)=\min\left(1,|b-a|,\frac{1}{\min(a,b)}\right),
\]
and, in particular,
\begin{equation}\label{eq:3}
d_\LL(\mathcal H_a,\mathcal H_b)\leq \min\left(1,|b-a|\right)=d_\LL(\mathcal H_{|b-a|},\mathcal H_0).
\end{equation}
\end{proposition}
Notice that the inequality \eqref{eq:3} expresses the more general fact that, for all $\lambda>0$ and for all $F$, $G\in\Delta^+$,
\[
d_\LL(\tau_\lambda F,\tau_\lambda G)\leq d_\LL(F,G),
\]
which is a consequence of the following property,
\[
\forall \lambda>0, \quad \left\{h>0, [0,h^{-1}[\subset A_{F,G}^h\right\} \subset \left\{h>0, [0,h^{-1}[\subset A_{\tau_\lambda F,\tau_\lambda G}^h\right\},
\]
where $\tau_\lambda F(t)=F(t-\lambda)$. This contraction property is an equality for the standard Levy metric while Proposition~\eqref{H} shows that it is not true for the modified version $d_\LL$.

\begin{proof}
In this proof, we will assume without loss of generality that $a<b$ and use the shortened notation
\[
A_{a,b}^h:=A_{\mathcal H_a,\mathcal H_b}^h=\left\{t\geq 0, \mathcal H_a(t)\leq\mathcal H_b(t+h)+h\right\},
\]
since in this case we have $\mathcal H_a\geq\mathcal H_b$. Notice that the inequality
\[
H_a(t)\leq\mathcal H_b(t+h)+h
\]
is immediate for $t\in[0,a]$, while if $t>a$ and since $h<1$, it is equivalent to
\[
\mathcal H_b(t+h)\geq 1-h >0,
\]
that is $t+h>b$. As a consequence,
\[
A_{a,b}^h=[0,a]\cup\left(]a,+\infty[\cap]b-h,+\infty[\right)\\ = [0,a]\cup]\max(a,b-h),+\infty[.
\]
We then have $2$ cases :
\begin{itemize}
\item If $a>1$, then for all $h\geq a^{-1}$, $[0,h^{-1}[\subset[0,a]\subset A_{a,b}^h$. In addition, if $h<a^{-1}$, then $[0,h^{-1}[\subset A_{a,b}^h$ if and only if $b-h\leq a$, that is $h\geq b-a$.
This leads to
\[
\left\{h>0, [0,h^{-1}[\subset A_{a,b}^h\right\} = [a^{-1},+\infty[\cup [b-a,+\infty[=[\min(a^{-1},b-a),+\infty[,
\]
hence in this case, $d_\LL(\mathcal H_a,\mathcal H_b)=\min(a^{-1},b-a)=\min(1,a^{-1},b-a)$.
\item If $a\leq 1$, we have $h^{1}>a$ for all $h\in]0,1[$, hence $[0,h^{-1}[\subset A_{a,b}^h$ if and only if $b-h\leq a$, that is if $h\geq b-a$. It follows that
\[
\left\{h>0, [0,h^{-1}[\subset A_{a,b}^h\right\} = [1,+\infty]\cup\left(]0,1[\cap[\min(1,b-a),+\infty[\right)=[\min(1,b-a),+\infty[,
\]
hence, in this case, $d_\LL(\mathcal H_a,\mathcal H_b)=\min(1,b-a)=\min(1,a^{-1},b-a)$.
\end{itemize}
This concludes the proof.
\end{proof}
%%%%%%%%%%%%%
%%%%%%%%%%%%%

%\begin{proof} To see the part $(1)$, suppose that $(K, D, \star)$ is compact as probabilistic metric space and let $(x_n)$ be a sequence. From Corollay \ref{passage},  $(K,\sigma_D)$ is a compact metric space and so there exists a subsequence $(x_{n_k})$ that converges to some $x\in K$ for the metric $\sigma_D$. Since $d_\LL(D(x_{n_k}, x),\mathcal{H}_0)\leq \sigma_D(x_{n_k}, x))$, we get that $d_\LL(D(x_{n_k}, x),\mathcal{H}_0)\to 0$, equivalently, $D(x_{n_k}, x)\,{\xrightarrow {\textnormal{w}}}\,\mathcal{H}_0$. Conversely, suppose that every sequence of $(K, D, \star)$ has a convergent subsequence. It follows that every sequence of the metric space $(K, \sigma_D)$ has a convergent subsequence. Thus, $(K, \sigma_D)$ is a compact metric space. Hence, $(K, D, \star)$ is compact as probabilistic metric space by Corollary \ref{passage}. The part $(2)$ can be proved in a similar way.
%\end{proof}
\section{Fixed point and contraction} \label{S3}
%Sehgal and Bharucha-Reid introduced in 1972 the notion of probabilistic $q$-contraction (q\in (0,1)$ in probabilistic metric space (see \cite{SB})
This section is divided on two subsections. In Subsection \ref{S31}, we give two new fixed point theorems and in Subsection \ref{S32}, we give some general examples of $k$-Lipschitz triangle functions constructed canonically from $k$-Lipschitz $t$-norms.
\subsection{Fixed point theorem} \label{S31}
Let us start from the following probabilistic notion of contraction introduced by Hicks  (see, \cite{HP}). 

\begin{definition} Let $(G,D,\star)$ be a probabilistic metric space. A map $f: G\to G$ is said to be a $C$-contraction  if there exists $q\in (0,1)$ such that for every $x, y \in G$ and every $t>0$
$$D(x,y)(t)> 1-t \Longrightarrow D(f(x),f(y))(qt)> 1-qt.$$
\end{definition}
\begin{lemma}  \label{C-contraction} A map $f: G\to G$ is  a $C$-contraction with constant $q$ iff for all $x, y \in G$,
$$d_\LL(D(f(x),f(y)), \mathcal{H}_0)\leq q d_\LL(D(x,y), \mathcal{H}_0).$$ 
\end{lemma}
\begin{proof}  From Lemma \ref{Topo}, we have that for every $x, y \in G$, $D(x,y)(t)> 1-t$ if and only if $d_\LL(D(x,y), \mathcal{H}_0) < t$.  For every $\varepsilon >0$, set $t_\varepsilon= d_\LL(D(x,y), \mathcal{H}_0)+\varepsilon>0$. Then, $D(x,y)(t_\varepsilon)> 1-t_\varepsilon$. Suppose that $f$ is a $C$-contraction, then we have that $D(f(x),f(y))(qt_\varepsilon)> 1-qt_\varepsilon$ which is equivalent to
	\[d_\LL(D(f(x),f(y), \mathcal{H}_0)\leq qt_\varepsilon=q(d_\LL(D(x,y), \mathcal{H}_0)+\varepsilon).\]
	 Sending $\varepsilon$ to $0$, we get $d_\LL(D(f(x),f(y), \mathcal{H}_0)\leq q d_\LL(D(x,y), \mathcal{H}_0)$. The converse is straightforward.
\end{proof}
Hicks proved that a $C$-contraction map in Menger space under the minimum $t$-norm $T_M(a,b)=\min(a,b)$ has a unique fixed point. We can find a extension of this result for generalised $C$-contraction in Menger space in \cite{HP}. We introduce the following new definition of contraction.
\begin{definition} Let $(G,D,\star)$ be a probabilistic metric space. Suppose that $\star$ is a continuous triangle function (hence uniformly continuous) and let $\omega_\star$ be a modulus of uniform continuity of $\star$. A map $f: G\to G$ is said to be a $\omega_\star$-contraction  if there exists $q\in (0,1)$ such that for every $x, y \in G$ 
$$\omega_\star[d_\LL(D(f(x),f(y)), \mathcal{H}_0)]\leq q \omega_\star[d_\LL(D(x,y), \mathcal{H}_0)].$$
\end{definition}
\begin{remark} Using Lemma \ref{C-contraction}, the notion of $\omega_\star$-contraction concides with the $C$-contraction, when the triangle function $\star$ is $k$-Lipschitz since in this case $\omega_\star(t)=kt$ for all $t\geq 0$ is a modulus of uniform continuity. Examples of $k$-Lipschitz triangle functions are given in Proposition \ref{Tri}. The original result of Hicks is a particular case corresponding to the $1$-Lipschitz triangle function $\star_{T_M}$.
\end{remark}

Using Theorem \ref{Metrizable} and the Ekeland variational principle, we give below an extension of the result of Hicks in probabilistic metric spaces which are not necesarily Menger spaces, where the triangle function $\star$ is continuous. Notice that this result seems to be new even in the non probabilistic setting.
\begin{theorem} \label{fix1} Let $(G,D,\star)$ be a probabilistic complete metric space, where $\star$ is continuous triangle function with modulus of uniform continuity $\omega_\star$. Let $f: G\to G$  be a $\omega_\star$-contraction with a constant of contraction $q\in (0, 1)$. Then, $f$ has a unique fixed point $x^*\in G$.
\end{theorem}
\begin{proof}  By assumption, we have for all $x$, $y\in G$
	\[
	\omega_\star(d_\LL(D(f(x),f(y)), \mathcal{H}_0))\leq q \omega_\star(d_\LL(D(x,y), \mathcal{H}_0)).
	\]
	Let us consider the function $\phi: (G,\sigma_D)\to \R$ defined by $\phi(x)=\omega_\star(d_\LL(D(x,f(x)), \mathcal{H}_0))$ and prove that $\phi$ is continuous. Indeed, for $x, y\in G$, from  the triangle inequality for $d_\LL$, the definition of $\sigma_D$ and Theorem \ref{Metrizable} we have
\begin{eqnarray*}
|d_\LL(D(x,f(x)), \mathcal{H}_0) - d_\LL(D(y,f(y)), \mathcal{H}_0)| &\leq&  d_\LL(D(x,f(x)), D(f(x),y)) \\
                                                                                                        && +d_\LL(D(f(x),y), D(y,f(y)))\\
                                                                                                     &\leq& \sigma_D(x,y) +\sigma_D(f(x),f(y))\\
                                                                                                     &\leq& \sigma_D(x,y) + \omega_\star(d_\LL(D(f(x),f(y)), \mathcal{H}_0))\\
                                                                                                     &\leq& \sigma_D(x,y) + q \omega_\star(d_\LL(D(x,y), \mathcal{H}_0)).
\end{eqnarray*}
From the continuity of $\omega_\star$, we have that $$d_\LL(D(x,y), \mathcal{H}_0)\leq \sigma_D(x,y) \to 0 \Longrightarrow \omega_\star(d_\LL(D(x,y), \mathcal{H}_0))\to 0.$$ Thus, from the above inequalities, the function $x\mapsto d_\LL(D(x,f(x)), \mathcal{H}_0) $ is continuous from $(G,\sigma_D)$ into $\R$, and by composing it with the uniformly continuous function $\omega_\star$, we get that $\phi$ is continuous.
\vskip5mm
Now, by the Ekeland variational principle \cite{Ek}  (since $(G,\sigma_D)$ is a complete metric space by Corollary \ref{passage}), let $\varepsilon> 0$ and $u\in G$ such that $\phi (u) \leq  \inf_G \phi+ \varepsilon$. Then,  for all $\lambda >0$ there exists $v\in G$ :

$(i)$ $\phi (v) \leq  \phi (u)$ ;

$(ii)$ $\sigma_D(u,v) \leq \lambda$;

$(iii)$ for all $x\in G$ $x\neq v$, $\phi(v)< \phi(x) +\frac{\varepsilon}{\lambda}\sigma_D(x,v)$.

Now, let us choose $\varepsilon <1-q$ and set $\lambda =1$. Using Theorem \ref{Metrizable} and $(iii)$, we have that  
\begin{eqnarray}\label{phi}
\phi(v)\leq \phi(x) +\varepsilon \omega_\star(d_\LL(D(x,v), \mathcal{H}_0)), \textnormal{ for all } x\in G.
\end{eqnarray}
We claim that $x^*:=f(v)$ is the unique fixed point of $f$.  Indeed, we have
\begin{eqnarray}\label{phi1}
\omega_\star(d_\LL(D(x^*,f(x^*)), \mathcal{H}_0))=\omega_\star(d_\LL(D(f(v),f(x^*)), \mathcal{H}_0))\leq q\omega_\star(d_\LL(D(v, x^*)). 
\end{eqnarray}
By (\ref{phi}) with $x^*$ and $v$, we have for all  $x\in G$
\begin{eqnarray}\label{phi2}
\omega_\star(d_\LL(D(v,x^*), \mathcal{H}_0))&=&\omega_\star(d_\LL(D(v,f(v)), \mathcal{H}_0))\nonumber\\
                                                                         &\leq& \omega_\star(d_\LL(D(x^*,f(x^*)), \mathcal{H}_0))+\varepsilon \omega_\star(d_\LL(D(x^*,v), \mathcal{H}_0)).
\end{eqnarray} 
Combining (\ref{phi1}) and (\ref{phi2}), we get 
$$\omega_\star(d_\LL(D(v,x^*), \mathcal{H}_0))\leq (q+\varepsilon) \omega_\star(d_\LL(D(v,x^*), \mathcal{H}_0)).$$
Since $q+\varepsilon <1$, we obtain that $\omega_\star(d_\LL(D(v,x^*), \mathcal{H}_0))=0$, which implies by Theorem \ref{Metrizable} that $\sigma_D(v,x^*)=0$, that is $x^*=v$. Thus, $f(x^*)=f(v)=:x^*$. The unicity of the fixed point $x^*$ is immediate from $q<1$.
\end{proof}

Applying the Banach fixed point we give the following extenstion of Hicks's result with an estimation of convergence of sequences $x_{n+1}=f(x_n)$. Note that we recover the Hicks's result with $\star=\star_{T_M}$ which is $k$-Lipschitz, with $k=1$.
\begin{theorem} \label{fix2}  Let $(G,D,\star)$ be a probabilistic complete metric space, where $\star$ is $k$-Lipschitz triangle function ($k\geq 1$). Let $f: G\to G$  be a $C$-contraction with a constant of contraction $q\in ]0, \frac 1 k[$. Then, $f$ has a unique fixed point $x^*\in G$. Moreover, every sequence $(x_n)$ of $G$ such that $x_{n+1}=f(x_n)$, satisfies: for every $t>0$
$$D(x_1,x_0)(t)> 1-t \Longrightarrow D(x_n,x^*)(\frac{k(kq)^n}{1-kq}t)> 1-\frac{k(kq)^n}{1-kq}t,$$
or equivalently,
$$d_\LL(D(x_n,x^*),\mathcal{H}_0)\leq \frac{k(kq)^n}{1-kq}d_\LL(D(x_1,x_0),\mathcal{H}_0).$$
In particular, $d_\LL(D(x_n,x^*),\mathcal{H}_0)\to 0$, when $n\to +\infty$.
\end{theorem}
\begin{proof} By Lemma \ref{C-contraction}, we have that $d_\LL(D(f(x),f(y), \mathcal{H}_0)\leq q d_\LL(D(x,y), \mathcal{H}_0)$, for all $x, y \in G$. Using Theorem \ref{Metrizable}, we get 
$$\sigma_D(f(x),f(y))\leq qk  \sigma_D(x,y).$$
Since $qk<1$, we can apply the Banach fixed point theorem in the complete metric space $(G,\sigma_D)$.  Thus, we obtain a unique fixed point $x^*$ such that, for all $n\in \N$
$$\sigma_D(x_n,x^*) \leq \frac{(kq)^n}{1-kq}\sigma_D(x_1,x_0) .$$ Using again Theorem \ref{Metrizable} (the $k$-Lipschitz part) we give 
$$d_\LL(D(x_n,x^*),\mathcal{H}_0)\leq \frac{k(kq)^n}{1-kq}d_\LL(D(x_1,x_0),\mathcal{H}_0),$$
which is equivalent by Lemma \ref{C-contraction} to: for all $t>0$
$$D(x_1,x_0)(t)> 1-t \Longrightarrow D(x_n,x^*)(\frac{k(kq)^n}{1-kq}t)> 1-\frac{k(kq)^n}{1-kq}t.$$
\end{proof}
%Note that $\star_{T_M}$ is $1$-Lipschitz where $T_M(u,v)=\min(u,v)$, so the above theorem applies to any $C$-contraction with $q\in (0,1)$ for $\star_{T_M}$ and in general for $1$-Lipschitz triangle function $\star$. Note also that in the above theorem, we have an estimation of convergence of any sequence $x_{n+1}=f(x_n)$ to the unique fixed point $x^*$, but only of $C$-contraction with constant $q\in (0, \frac 1 k)$. 
%Using Ekeland variational principle, we get in the following fixed point theorem witout estimation of convergence but for any $C$-contraction with constant $q\in (0, 1)$, where the triangle function $\star$ is $k$-Lipschitz.
\subsection{$k$-Lipschitz triangle function} \label{S32}
One of the standard way to construct triangle function goes in the following way. We refer to \cite{HP} for more details.
\begin{definition} We denote by $\mathcal{L}$ the set of all binary operators $L$ on $[0,+\infty[$ which satisfy the following conditions:

$(i)$ $L$ maps $[0,+\infty[^2$ to $[0,+\infty[$

$(ii)$ $L$ is non-deceasing in both coordinate

$(iii)$ $L$ is continuous on $[0,+\infty[^2$.
\end{definition}
For a $t$-norm $T$, we define the operation $\star_{T,L}$ from $\Delta^+\times \Delta^+$ to $\Delta^+$ as follows: for every $F,G\in \Delta^+$ and every $t\geq 0$
\begin{eqnarray*} 
(F\star_{T,L} G)(t)= \sup_{ L(u,v)=t} T(F(u),G(v)).
\end{eqnarray*} 
In the speciale case where $L(u,v)=u+v$ we obtain $\star_{T,L}=\star_T$.
\begin{theorem} \textnormal{ (\cite[Theorem 2.15]{HP})} if $T$ is a left-continuous $t$-norm and $L\in \mathcal{L}$ is commutative, associative, has $0$ as identity and satisfy the condition  
$$\textnormal{if }u_1<u_2 \textnormal{ and } v_1<v_2 \textnormal{ then } L(u_1,v_1) < L(u_2,v_2),$$
then, $\star_{T,L}$ is a triangle function.
\end{theorem}
The above theorem works for example with $L(u,v):=L_+(u,v)=u+v$ or $L(u,v):=L_{M}(u,v)=\max(u,v)$.
\vskip5mm
Another way to construct a triangle function from a $t$-norm $T$ is the use the $t$-conorm $T^*(u,v)=1-T(1-u,1-v)$ as follows : for every $F,G\in \Delta^+$ and for every $s>0$
$$(F\star_{T^*} G)(s)= \inf_{u+v=s} T^*(F(u),G(v)).$$
Recall that a $t$-norm $T:[0,1]\times [0,1]\to [0,1]$ is $k$-Lipschitz if there exists $k\in [0,+\infty[$ such that, for all $a,b,c,d\in [0,1]$, we have 
$$|T(a,b)-T(c,d)|\leq k (|a-c|+|b-d|).$$
Since $T(x,1)=x$, we necessarily have that $k\geq 1$. Note also that the minimum $t$-norm $T_M(a,b):=\min (a,b)$ is $1$-Lipschitz. Other examples of $k$-Lipschitz $t$-norms are studied in \cite{Me1,Me2,Me3}.  In order to give examples of $k$-Lipschitz triangle functions in Proposition \ref{Tri}, we need the following lemma.
\begin{lemma} \label{Tri0} Let $T:[0,1]\times [0,1]\to [0,1]$ be a $k$-Lipschitz $t$-norm. Then, for every $(a, b), (c,d)\in [0,1]\times [0,1]$ and every $h\in [0,+\infty[$ such that $a\leq c+h_1$ and $b\leq d+h_2$ we have that 
\begin{eqnarray*} 
T(a,b)-T(c,d)\leq k (h_1+h_2)\\
T^*(a,b)-T^*(c,d)\leq k (h_1+h_2).
\end{eqnarray*}
\begin{proof} Four cases are discussed.

{\bf case 1.} If $a\leq c$ and $b\leq d$. In this case, since $T$ is a t-norm, then
\[T(a,b)-T(c,d)\leq 0 \leq k(h_1+h_2).\]

{\bf case 2.} If $a\leq c$ and $b\geq d$.  In this case, since $T$ is a t-norm, then $T(a,b)\leq T(c,b)$ and so  since it is $k$-Lipschitz we have that 
\[T(a,b)-T(c,d)\leq T(c,b) -T(c,d)\leq k|b-d|=k(b-d)\leq kh_1\leq k(h_1+h_2).\]

{\bf case 3.} If $a\geq c$ and $b\leq d$. This case is similar to case 2.

{\bf case 4.} If $a\geq c$ and $b\geq d$. In this case, since $T$ is $k$-Lipschitz we have that
\[T(a,b)-T(c,d)\leq  k (|a-c|+|b-d|)=k(a-c+b-d)\leq k(h_1+h_2).\]

The case of $T^*$ comes easily from the case of $T$.
\end{proof}
\end{lemma}
In the following proposition, we consider the cases where $L(u,v):=L_+(u,v)=u+v$ and $L(u,v):=L_M(u,v)=\max(u,v)$.
\begin{proposition} \label{Tri} Let $T:[0,1]\times [0,1]\to [0,1]$ be a $k$-Lipschitz $t$-norm ($k\geq 1$). Then, the triangle functions $\star_T$; $\star_{T,L_M}$; $\star_{T^*}$ and $\T$ are $k$-Lipschitz , where for all $F, G\in \Delta^+$ and all $s>0$
\begin{eqnarray*} 
(F\star_T G)(s)&=&\sup_{u+v=s} T(F(u), G(v)),
\end{eqnarray*}
\begin{eqnarray*} 
(F\star_{T^*} G)(s)= \inf_{u+v=s} T^*(F(u),G(v)),
\end{eqnarray*}
\begin{eqnarray*} 
(F\star_{T,L_M} G)(s)=\sup_{\max(u,v)=s} T(F(u), G(v)),
\end{eqnarray*}
\begin{eqnarray*} 
\T( F, G)(s)= T(F(s),G(s)).
\end{eqnarray*}
\end{proposition} 
\begin{proof} We give the prove for $\star_T$, the technique is similar for the other triangle functions. Let $F,F', G, G'\in \Delta^+$. Let $h_1, h_2 > 0$ be such that
\begin{equation}\label{eq:4}
[0,h_1^{-1}[\subset A_{F,F'}^{h_1}\cap A_{F',F}^{h_1} \text{ and } [0,h_2^{-1}[\subset A_{G,G'}^{h_2}\cap A_{G',G}^{h_2},
\end{equation}
meanning that for all $t\in ]0,h_1^{-1}[$ and all $t'\in ]0,h_2^{-1}[$ we have:
\begin{eqnarray*}
0\leq F(t) \leq  F'(t+h_1) + h_1 \\
%\leq  F'(t+\max(h_1,h_2)) + \max(h_1,h_2)\\
0\leq F'(t) \leq  F(t+h_1) + h_1 \\
%\leq F(t+\max(h_1,h_2))+ \max(h_1,h_2)\\
0\leq G(t') \leq  G'(t'+h_2) + h_2\\
%\leq G'(t+\max(h_1,h_2))+\max(h_1,h_2)\\
0\leq G'(t') \leq  G(t'+h_2) + h_2.
%\leq G(t'+\max(h_1,h_2))+\max(h_1,h_2).
\end{eqnarray*}
%Let $0<c\leq \max(h_1,h_2)$ such that for all $t\in (0,c^{-1})$ 
Thus, combining the first and the third (resp. the second and the fourth) inequalities, and using  Lemma \ref{Tri0}, we have that for every $u,v\in ]0, \max(h_1,h_2)^{-1}[$($\subset ]0,\min(h_1^{-1},h_2^{-1})[$) 
\begin{eqnarray*}
 0\leq T(F(u),G(v)) &\leq& T( F'(u+h_1) ,G'(v+h_2)) +k(h_1+h_2)\\
 0\leq T(F'(u),G'(v))&\leq& T( F(u+h_1) ,G(v+h_2)) +k(h_1+h_2)
\end{eqnarray*}
Let $s\in ]0, \max(h_1,h_2)^{-1}[$, taking the supremum over $0\leq u, v$ such that $u+v=s$ in the above inequalities with the fact that $k\geq 1$, we get
\begin{eqnarray*}
0\leq (F\star_{T} G)(s) &\leq&  (F'\star_{T} G')  (s+(h_1+h_2)) + k(h_1+h_2)\\
                                &\leq& (F'\star_{T} G')  (s+k(h_1 +h_2)) + k(h_1+h_2)\\
0\leq (F'\star_{T} G')(s) &\leq&  (F\star_{T} G)  (s+(h_1+h_2)) + k(h_1+h_2) \\
                                 &\leq&  (F\star_{T} G)  (s+k(h_1+h_2)) + k(h_1+h_2)
\end{eqnarray*}
This shows that for all $h_1, h_2 > 0$ satisfying \eqref{eq:4}, we have
\[ d_\LL( F\star_{T} G,  F'\star_{T} G')\leq k(h_1+h_2).\]
Thus, taking the infinimum over $h_1$ and $h_2$, we get
$$d_\LL( F\star_{T} G,  F'\star_{T} G')\leq k (d_\LL( F,  F')+d_\LL( G,  G')).$$
%Similarly, using $(E1)$ and $(E2)$, we get
%$$d_\LL( F\star_{T,L_M} G,  F'\star_{T,L_M} G')\leq k (d_\LL( F,  F')+d_\LL( G,  G')).$$
%$$d_\LL( \T(F,G),  \T(F',G'))\leq k (d_\LL( F,  F')+d_\LL( G,  G')).$$
\end{proof}

%\begin{proof} $(1)$ Suppose that $(G, D, \star)$ is a probabilistic complete metric space. Let $(x_n)$ be a Cauchy sequence in $(G,\sigma_D)$, then from the inequality at the left in Theorem \ref{Metrizable} we get that $(x_n)$ is Cauchy sequence in $(G, D, \star)$ which is complete. Thus there exists $x\in G$ such that $d_\LL(D(x_n,x),\mathcal{H}_0)\to 0$. Now, using the inequality at the right in Theorem \ref{Metrizable}, we see that $(x_n)$ converges to $x$ for the metric $\sigma_D$. Hence $(G,\sigma_D)$ is a complete metric space. For the converse, suppose that $(G,\sigma_D)$ is a complete metric space. Let $(x_n)$ be a Cauchy sequence in $(G, D, \star)$. Using the inequality at the right of  Theorem \ref{Metrizable}, we get that $(x_n)$ is also a Cauchy sequence in $(G,\sigma_D)$ and so converges to some $x\in G$ for $\sigma_D$ since $(G,\sigma_D)$ is complete. The  inequality at the left of Theorem \ref{Metrizable} implies that $(x_n)$ also converges to $x$  in $(G, D, \star)$. Thus, $(G, D, \star)$ is a probabilistic complete metric space.

%$(2)$  Suppose that $(G,\sigma_D)$ is compact, then we use the part $(i)$ of Theorem \ref{Metrizable} to see that $(G, D, \star)$ is compact as probabilistic metric space. The converse follows from the part $(ii)$ of Theorem \ref{Metrizable}. 
%\end{proof} 
%%%%%%%%%%%%%
%%%%%%%%%%%%%
\section{Probabilistic Arzela-Ascoli type theorem}\label{S4}
This section is divided on two subsections. In subsection \ref{S41} we give some general definitions of probabilistic function spaces and in subsection \ref{S42}, we give the main result of this section, a probabilistic Arzela-Ascoli type theorem.
\subsection{The space of continuous and functions} \label{S41}
We are going to define continuity of functions defined from a probabilistic metric space $(G, D, \star)$ to a (deterministic) metric space  $(F,d_F)$.
 
\begin{definition}  Let $(G, D, \star)$ be a probabilistic metric space, $(F,d_F)$ be a metric space and let $f$ be a function $f : (G, D, \star) \longrightarrow (F,d_F)$. We say that $f$ is (probabilistic) continuous at $z\in G$ if $d_F(f(z_n)\, f(z))\to 0$  whenever $D(z_n,z)\,{\xrightarrow {\textnormal{w}}}\,\mathcal{H}_0$ (equivalently $d_\LL(D(z_n,z),\mathcal{H}_0)\to 0$). We say that $f$ is continuous if $f$ is continuous at each point $z\in G$.
\end{definition}
 By $C_\star(G,F)$ we denote the space of all (probabilistic) continuous  functions $f : (G, D, \star) \longrightarrow (F,d_F)$. By $C(G,F)$ we denote the space of all (deterministic) continuous functions $f : (G, \sigma_D) \longrightarrow (F,d_F)$. We both equip the spaces $C_\star(G,F)$ and $C(G,F)$  with the uniform metric 
\begin{eqnarray*}
d_{\infty}(f,g):=\sup_{x\in G} d_F(f(x), g(x))
\end{eqnarray*}
As in the standard case, the completeness of $C_\star(G,F)$ only relies on the completeness of the arrival space.
\begin{proposition} \label{complete} Let $(G, D, \star)$ be a probabilistic metric space (here $\star$ is not assumed to be continuous) and $(F,d_F)$ be a complete metric space. Then, the space $(C_\star(G,F), d_{\infty})$ is a complete metric space.
\end{proposition}
\begin{proof} Let $(f_n)$ be a Cauchy sequence in $(C_\star(G,F), d_{\infty})$. In particular, for each $x\in G$, $(f_n(x))$ is Cauchy in $(F, d_F)$ which is complete. Thus, there exists a function $f : G \longrightarrow F$ such that the sequence $(f_n)$ pointwise converges to $f$ on $G$. It is easy to see that in fact $(f_n)$ uniformly converges to $f$, since it is Cauchy sequence in $(C_\star(G,F), d_{\infty})$. We need to prove that $f$ is a continuous function from  $(G, D, \star)$ into $(F,d_F)$. Let $x\in G$ and $(x_k)$ be a sequence such that $d_\LL(D(x_k,x), \mathcal{H}_0)\longrightarrow 0$, when $k\longrightarrow +\infty$. 
For all $\varepsilon >0$, there exists $N_\varepsilon\in \N$ such that 
\begin{eqnarray}\label{eqq1}
n \geq N_\varepsilon \Longrightarrow d_{\infty} (f_n, f):=\sup_{x\in G} d_F(f_n(x),f(x))\leq \varepsilon
\end{eqnarray}
Using the continuity of $f_{N_\varepsilon}$, we have that there exists $\eta(\varepsilon)>0$ such that
\begin{eqnarray}\label{eqq2}
d_\LL(D(x_k,x),\mathcal{H}_0) \leq \eta(\varepsilon) \Longrightarrow d_F(f_{N_\varepsilon}(x_k),f_{N_\varepsilon}(x)) \leq \varepsilon
\end{eqnarray}
Using (\ref{eqq1}) and (\ref{eqq2}), we have that  
\begin{eqnarray*}
d_F(f(x_k),f(x))&\leq& d_F(f(x_k),f_{N_\varepsilon}(x_k))+d_F(f_{N_\varepsilon}(x_k),f_{N_\varepsilon}(x))+ d_F(f_{N_\varepsilon}(x),f(x))\\
                &\leq& 3\varepsilon
\end{eqnarray*}
This shows that $f$ is continuous on $G$. Finally, we proved that every Cauchy sequence $(f_n)$ uniformly converges to a continuous function $f$. In other words, the space $(C_\star(G,F), d_{\infty})$ is complete.
\end{proof}
Since, for all $x$, $y\in G$,
\[
d_{\mathcal L}(D(x,y),\mathcal H_0)\leq \sigma_D(x,y),
\]
We have in general that $C_\star(G,F)\subset C(G,F)$. Assuming the continuity of the triangle function $\star$, we obtain the equality.
\begin{proposition} \label{compar} Let $(G, D, \star)$ be a probabilistic metric space and $(F,d)$ be a metric space. Suppose that $\star$ is continuous. Then, we have $C_\star(G,F)=C(G,F)$. In the case where $(F,d_F)=(\Delta^+,d_{\mathcal{L}})$, we also have that $Lip^1_\star(G,\Delta^+)\subset C_\star(G,\Delta^+)$.
\end{proposition}
\begin{proof} Thanks to Theorem~\ref{Metrizable}, we have, for all $x$, $y\in G$,
	\[
	\sigma_D(x,y)\leq \omega_\star\left(d_{\mathcal L}(D(x,y)),\mathcal H_0\right),
	\]
	providing the equality between $C_\star(G,F)$ and $C(G,F)$. For the second part of the statement, we use Lemma~\ref{equi} to ensure that, for any $f\in Lip^1_\star(G,\Delta^+)$ and for any $x$, $y\in G$,
	\[
	d_{\mathcal L}(f(x),f(y)) \leq \sup_{g\in Lip_\star^1(G,\Delta^+)}\left[d_{\mathcal L}(g(x),g(y))\right]\leq\omega_\star\left(d_{\mathcal L}(D(x,y)),\mathcal H_0\right),
	\]
	which gives the conclusion.
\end{proof}
\begin{remark} We do not know if $Lip^1_\star(G,\Delta^+)\subset C_\star(G,\Delta^+)$ when $\star$ is not continuous.
\end{remark}
\subsection{Arzela-Ascoli type theorem for the space $ Lip^1_\star(K,\Delta^+)$} \label{S42}
The following proposition gives a canonical way to build probabilistic $1$-Lipschitz maps from $(G,D,\star)$ into $\Delta^+$.
\begin{definition} A triangle function $\star$ is  said to be sup-continuous (see for instance \cite{C3}) if for all nonempty set $I$ and all familly $(F_i)_{i\in I}$ of distributions in $\Delta^+$ and all $L\in\Delta^+$, we have $$\sup_{i\in I} (F_i\star L)=\sup_{i\in I}(F_i)\star L.$$
\end{definition}
\begin{proposition} Let $(G, D, \star)$ be a probabilistic metric space such that $\star$ is sup-continuous. Let $f : (G,D,\star) \longrightarrow \Delta^+$ be any map and $A$ be any no-empty subset of $G$. Then, the map $\tilde{f}_A(x):=\sup_{y\in A} [ f(y) \star D(x,y)]$, for all $x\in G$ is a probabilistic $1$-Lipschitz map and we have $\tilde{f}_A(x)\geq f(x)$, for all $x\in A$.
\end{proposition}
\begin{proof} The proof is similar to the standard inf-convolution construction. The fact that $\tilde{f}_A(x)\geq f(x)$ for all $x\in A$ is immediate from the definition of $\tilde{f}_A$. Let us now prove that it is probabilistic $1$-Lipschiptz. Let $x$, $y\in G$. Then, for all $z\in A$, we have
	\begin{multline*}
		\tilde{f}_A(y) = \sup_{z\in A}\left[f(z)\star D(y,z)\right]\geq f(z)\star D(y,z)\\
		\geq f(z)\star \left(D(y,x)\star D(x,z)\right)=\left(f(z)\star D(x,z)\right)\star D(y,x).
	\end{multline*}
	We get the conclusion by taking the supremum with respect to $z\in A$ and using the  sup-continuity of $\star$.
	\end{proof}
%%%%%%%%%%%%%%%%%%%%%%
%%%%%%%%%%%%%%%%%%%%%%
 Let us now recall the following result from \cite{Ba}.
\begin{proposition} (\cite[Proposition 3.5]{Ba}) \label{util}
Let $(F_n), (L_n), (K_n) \subset (\Delta^+,\star)$. Suppose that 
\begin{itemize}
\item[(a)] the triangle function $\star$ is continuous,
\item[(b)] $F_n\,{\xrightarrow {\textnormal{w}}}\, F$, $L_n\,{\xrightarrow {\textnormal{w}}}\, L$ and $K_n\,{\xrightarrow {\textnormal{w}}}\, K$. 
\item[(c)] for all $n\in \N$, $F_n \star L_n \leq K_n$.
\end{itemize}
Then, $F\star L\leq K$.
\end{proposition}
%\begin{lemma} \label{Lip} Let $(G, D, \star)$ be a probabilistic metric space. Let $(f_n)$ be a sequence of $1$-Lipschitz maps. Suppose that there exists a function $f$ defined on $G$ such that $f_n(x) \,{\xrightarrow {\textnormal{w}}}\, f(x)$, when $n \longrightarrow +\infty$, for all $x\in G$. Then, $f$ is (probabilistic) $1$-Lipschitz on $G$.
%\end{lemma}
%\begin{proof} Since each $f_n$ is a $1$-Lipschitz map, we have for all $x, y \in L$ and for all $n\in \N$:
%\begin{eqnarray*}
%D(x,y)\star f_n(x) \leq f_n(y)
%\end{eqnarray*}
%Using Proposition \ref{util}, we get that for all $x, y \in L$
%\begin{eqnarray*}
%D(x,y)\star f(x) \leq f(y)
%\end{eqnarray*}
%In other words, $f$ is $1$-Lipschitz maps on $L$.
%\end{proof}
%%%%%%%%%%%%%%%%%%%%

%%%%%%%%%%%%%%%%%%%%%%
\begin{lemma} \label{unif} Let $(K, D, \star)$ be a probabilistic compact metric space and $(f_n)$ be a sequence of probabilistic $1$-Lipschitz maps. Suppose that there exists a function $f$ defined from $K$  into $\Delta^+$ such that, for all $x\in K$, $d_\LL(f_n(x),f(x))\longrightarrow 0$, as $n\to +\infty$. Then, $f$ is (probabilistic) $1$-Lipschitz on $K$ and $(f_n)$ converges uniformly to $f$, that is, $d_{\infty}(f_n,f)\longrightarrow 0$, as $n \to +\infty$.
\end{lemma}
\begin{proof} Since each $f_n$ is a $1$-Lipschitz map, we have for all $x, y \in L$ and for all $n\in \N$:
\begin{eqnarray*}
D(x,y)\star f_n(x) \leq f_n(y)
\end{eqnarray*}
Using Proposition \ref{util}, we get that for all $x, y \in L$
\begin{eqnarray*}
D(x,y)\star f(x) \leq f(y)
\end{eqnarray*}
In other words, $f$ is $1$-Lipschitz maps on $L$ (Note that up to now, we have not needed to use the compactness of $K$).

Now,  let $\varepsilon >0$ and, using Lemma \ref{equi}, let $\eta(\varepsilon)$ be the uniform modulus of equicontinuity for the set $Lip_\star^1(K,\Delta^+)$. Since $(K, D, \star)$ is compact, there exists a finite set $A$ such that 
$K=\cup_{a\in A} N_a(\eta(\varepsilon))$. 
Since $d_\LL(f_n(a),f(a))\longrightarrow 0$, as $n\to +\infty$ for all $a\in A$. Then, for each $a\in A$, there exists $P_a\in \N$ such that 
\begin{eqnarray*}
n\geq P_a \Longrightarrow d_\LL(f_n(a),f(a)) & \leq & \varepsilon
\end{eqnarray*}
Since $A$ is finite, we have that
\begin{eqnarray*}
n\geq \max_{a\in A} P_a \Longrightarrow \sup_{a\in A} d_\LL(f_n(a),f(a)) & \leq & \varepsilon
\end{eqnarray*}
Thus, for all $x\in K=\cup_{a\in A} N_a(\eta(\varepsilon))$, there exists $a\in A$ such that $x\in  N_a(\eta(\varepsilon))$ and so we have that for all $n\geq \max_{a\in A} P_a$ :
\begin{eqnarray*}
d_\LL(f_n(x),f(x)) &\leq& d_\LL(f_n(x),f_n(a)) + d_\LL(f_n(a),f(a)) + d_\LL(f(a),f(x))\\
                                                         &\leq& 3\varepsilon.
\end{eqnarray*} 
In other words, 
\begin{eqnarray*}
n\geq \max_{a\in A} P_a \Longrightarrow d_{\infty}(f_n,f):=\sup_{x\in K} d_\LL(f_n(x),f(x)) &\leq& 3\varepsilon
\end{eqnarray*}
\end{proof}
%%%%%%%%%%%%%%%%%%%%%%%
%
%
We give now our main result of this section. For the classical Arzela-Ascoli theorem we refer to the book of L.~Schwartz, Analyse I, {\it ''Théorie des ensembles et Topologie''}, page 346.
\begin{theorem} \label{AA} Let $(K, D, \star)$ be a probabilistic complete metric space such that $\star$ is continuous. Then, the following assertions are equivalent. 
\begin{enumerate}
\item $(K, D, \star)$ is compact.
\item The metric space $(Lip^1_\star(K,\Delta^+),d_{\infty})$ is compact (or equivalently, $Lip^1_\star(K,\Delta^+)$ is a compact subset of $(C_\star(K,\Delta^+), d_{\infty})=(C(K,\Delta^+), d_{\infty})$).
\end{enumerate}
\end{theorem}
\begin{proof} $ \bullet\ (1) \Longrightarrow (2)$ Suppose that $(K, D, \star)$ is compact, equivalently $(K,\sigma_D)$ is compact by Corollary \ref{passage}.  Using Lemma \ref{equi} and Theorem \ref{Metrizable},  the set $Lip_\star^1(G,\Delta^+)$ is uniformly  equicontinuous with respect to the metric $\sigma_D$. Moreover, $(\Delta^+,d_\LL)$ is compact, hence $Lip_\star^1(G,\Delta^+)$ is relatively compact in $(C(K,\Delta^+), d_{\infty})$ by Arzela-Ascoli theorem. On the other hand, by Lemma \ref{unif},  the set $Lip_\star^1(G,\Delta^+)$ is closed in $(C(G,\Delta^+),d_{\infty})$. Hence it is compact.
	
$\bullet\ (2)  \Longrightarrow (1)$ Suppose that $(Lip^1_\star(K,\Delta^+),d_{\infty})$ is compact. Let $(x_n)$ be a sequence of $K$. We need to prove that $(x_n)$ has a convergent subsequence. Consider the sequence $(\delta_{x_n})$ of $1$-Lipschitz maps, defined by $\delta_{x_n}: x\mapsto D(x_n,x)$ for each $n\in \N$. By assumption, there exists a subsequence $(\delta_{x_{\varphi(n)}})$ that converges uniformly to some $1$-Lipschitz map, in particular it is a Cauchy sequence. In other words, we have 
$$\lim_{p,q\longrightarrow +\infty} \sup_{x\in K} d_\LL(\delta_{x_{\varphi(p)}}(x),\delta_{x_{\varphi(q)}}(x))=0.$$
In particular we have 
$$\lim_{p,q\longrightarrow +\infty} d_\LL(\delta_{x_{\varphi(p)}}(x_{\varphi(q)}),\mathcal{H}_0)=0,$$
or equivalently,
$$\lim_{p,q\longrightarrow +\infty} d_\LL(D(x_{\varphi(p)},x_{\varphi(q)}),\mathcal{H}_0)=0.$$
This shows that the sequence $(x_{\varphi(n)})$ is Cauchy in $(K,\sigma_D)$ (see Theorem \ref{Metrizable}). Thus, the sequence $(x_{\varphi(n)})$ converges to some point $x\in K$ for the metric $\sigma_D$, since $(K,\sigma_D)$  is complete. Hence, $(K,\sigma_D)$ is compact, equivalently, $(K, D, \star)$ is compact. This ends the proof.
\end{proof}
%\section{Conclusion}

%%%%%%%%%%%%%%%%%%%%%%%%%%%
%%%%%%%%%%%%%%%%%%%%%%%%%%%%%

\bibliographystyle{amsplain}

\providecommand{\bysame}{\leavevmode\hbox to3em{\hrulefill}\thinspace}
\providecommand{\MR}{\relax\ifhmode\unskip\space\fi MR }
\providecommand{\MRhref}[2]{%
  \href{http://www.ams.org/mathscinet-getitem?mr=#1}{#2}
}
\providecommand{\href}[2]{#2}
\begin{thebibliography}{}

\end{thebibliography}


\begin{thebibliography}{999}
\bibitem{Ba} M. Bachir, \textit{The Space of Probabilistic $1$-Lipschitz map}, Aequationes Math. (2019) 1-29.
%\bibitem{Ba2} M. Bachir, \textit{A Banach-Stone type theorem for Invariant Metric Groups}, Topology Appl. 209 (2016) 189-197.
\bibitem{C3} S. Cobzas, \textit{Completeness with respect to the probabilistic Pompeiy-Hausdorff metric} Studia Univ. "Babes-Bolyai", Mathematica, Volume LII, Number 3, (2007) 43-65.
\bibitem{Ek} I. Ekeland \textit{On the variational principle}, J. Math. Anal. Appl. 47 (1974), 324-353.
\bibitem{HP} O. Had\v{z}i\'c and E. Pap \textit{Fixed Point Theory in Probabilistic Metric Space}, vol. 536 of Mathematics and Its Applications, Kluwer Academic Publishers, Dordrecht, The Netherlands, 2001. 
%\bibitem{HP1} O. Had\v{z}i\'c and E. Pap \textit{On the Local Uniqueness of the Fixed Point of the Probabilistic q-Contraction in Fuzzy Metric Spaces} Published by Faculty of Sciences and Mathematics, University of Ni\v{s}, Serbia (2017), 3453-3458, https://doi.org/10.2298/FIL1711453H
\bibitem{KMP2} E. P. Klement, R. Mesiar, E. Pap, \textit{Triangular norms I: Basic analytical and algebraic properties}, Fuzzy Sets and Systems 143 (2004), 5-26.
\bibitem{KMP3} E. P. Klement, R. Mesiar, E. Pap,  Triangular Norms. Kluwer, Dordrecht (2000).
\bibitem{Me1} A. Mesiarov\'a: {\it Triangular norms and k-Lipschitz property}. In: Proc. EUSFLAT-LFA Conference, Barcelona 2005, pp. 922-926.
\bibitem{Me2} A. Mesiarov\'a, \textit{$k$-$lp$-Lipschitz $t$-norms}, International Journal of Approximate Reasoning 46 (2007) 596-604.
\bibitem{Me3} A. Mesiarov\'a, \textit{ Lipschitz continuity of triangular norms}, in: B. Reusch (Ed.), Computational Intelligence, Theory and Applications (Proc. 9th Fuzzy Days in Dortmund), Springer, Berlin, 2006, pp. 309-321.
\bibitem{MON} A. Mbarki, A. Ouahab, R. Naciri,  \textit{On Compactness of Probabilistic Metric Space}, Applied Mathematical Sciences Volume(8), (2014) 1703-1710.
%\bibitem{M1} K. Menger, \textit{Statistical metrics}, Proc. Nat. Acad. of Sci. U.S.A. 28 (1942), 535-537.
%\bibitem{M2} K. Menger, \textit{Probabilistic geometry}, Pro. Nat. Acad. of Sci. U.S.A. 37 (1951), 226-229.
%\bibitem{M3} K. Menger, \textit{G\'eom\'etrie g\'en\'erale} (Chap. VII), Memorial des Sciences Mathematiques,
No. 124, Paris 1954. 
\bibitem{MN} B. Morrel, J. Nagata: \textit{Statistical metric spaces as related to topological spaces},  Gen. Top. Appl. , 9 , (1978 233-237 .
%\bibitem{S-P.S} S. Saminger-Platz, C. Sempi \textit{A primer on triangle functions I}, Aequationes Math. 76 (2008) 201-240.

\bibitem{S.S} B. Schweizer and A. Sklar,  \textit{Probabilistic metric spaces}, North-Holland Series in Probability and Applied Mathematics, North-Holland Publishing Co., New York, 1983.
%\bibitem{SS0} B. Schweizer and A. Sklar, \textit{Espaces metriques aleatoires}, C. R. Acad. Sci., Paris
247 (1958), 2092-2094.
\bibitem{SS1} B. Schweizer, A. Sklar \textit{Statistical metric spaces}, Pacific. J. Math. 10, (1960) 313-334.
%\bibitem{SS2} B. Schweizer, A. Sklar \textit{Triangle Inequalities in a Class of Statistical Metric Spaces}, J. London
Math. Soc. 38 (1963), 401-406.
\bibitem{SS3} B. Schweizer, A. Sklar and E. Thorp, \textit{The metrization of statistical metric svaces}, Pacific J. Math. 10 (1960), 673-675.
%\bibitem{SS4} B. Schweizer and A. Sklar, \textit{probabilistic metric spaces}, North-Holland, New York(1983).
%\bibitem{SST} B. Schweizer, A. Sklar, E. Thorp: The metrization of SM- spaces. Pacific J. Math. 10 (1960) 673-75.
\bibitem{SHE0} H. Sherwood, \textit{Complete probabilistic metric spaces and random variables generated spaces}, Ph.D. Thesis, University of Arizona (1965).
%\bibitem{SHE} H. Sherwood \textit{On E-Spaces and their Relation to Other Classes of Probabilistic Metric Spaces} J. London
Math. Soc. s1-44, (1969) 441-448 
%\bibitem{SHE1} H. Sherwood \textit{Complete probabilistic metric spaces} Z. Wahrscheinlichkeitstheorie und Verw.
Gebiete 20, (1971) 117-128.
\bibitem{DS} D. A. Sibley \textit{A metric for weak convergence of distribution functions}, Rocky Mountain J. Math.
l (l971) 427-430.
\end{thebibliography}

\end{document}